\documentclass[12pt]{amsart}

\usepackage{amsmath}
\usepackage{amssymb}
\usepackage{graphicx}
\usepackage[colorlinks=true, allcolors=blue]{hyperref}
\usepackage[shortlabels]{enumitem}
\usepackage[foot]{amsaddr}
\newtheorem{theorem}{Theorem}
\newtheorem{proposition}[theorem]{Proposition}
\newtheorem{lemma}[theorem]{Lemma}

\newtheorem{definition}[theorem]{Definition}

\newtheorem{example}[theorem]{Example}

\numberwithin{theorem}{section}

\makeatletter
\g@addto@macro{\endabstract}{\@setabstract}
\newcommand{\authorfootnotes}{\renewcommand\thefootnote{\@fnsymbol\c@footnote}}%
\makeatother
\begin{document}
\begin{center}
 \LARGE 
Unshuffling a deck of cards 
\par 
\bigskip
\normalsize
\authorfootnotes
Cornelia A. Van Cott and Katie Wang
\end{center}
\subjclass[2000]{Primary 20B35; Secondary ⟨secondary classifications⟩}

\begin{abstract}
We investigate the mathematics behind unshuffles, a type of card shuffle closely related to classical perfect shuffles. To perform an unshuffle, deal all the cards alternately into two piles and then stack the one pile on top of the other. There are two ways this stacking can be done (left stack on top or right stack on top), giving rise to the terms left shuffle ($L$) and right shuffle ($R$), respectively. We give a solution to a generalization of Elmsley's Problem (a classic mathematical card trick) using unshuffles for decks with $2^k$ cards. We also find the structure of the permutation groups  $\langle L, R \rangle$ for a deck of $2n$ cards for all values of $n$. We prove that the group coincides with the perfect shuffle group unless $n\equiv 3 \pmod 4$, in which case the group $\langle L, R \rangle$ is equal to $B_n$, the group of centrally symmetric permutations of $2n$ elements, while the perfect shuffle group is an index 2 subgroup of $B_n$.
\end{abstract}

\section{Introduction}
A well known card shuffling technique in the world of magicians is the so-called {\em perfect shuffle} (also known as the {\em faro shuffle}). Not only are perfect shuffles used by magicians to do card tricks, but gamblers have used these shuffles to cheat at card games since the 1800's~\cite{expose,braue,Green, Jordan}. 
As with all card shuffling techniques, perfect shuffles can be considered from a mathematical perspective as a permutations of the set of cards. From this point of view, we can better understand the characteristics, limitations, and properties of these shuffles. 

Here is how to do a perfect shuffle. Take a deck of $2n$ cards and split the deck exactly in half. Next, perfectly interlace the cards from the two stacks together. There are two ways that this interlacing can be done. An {\em out shuffle} interlaces the stacks in such a way that the top and bottom cards from the original deck remain on the top and bottom, respectively. An {\em in shuffle} interlaces the cards so that the top and bottom cards from the original deck become the second and second-to-last card, respectively. See Figure~\ref{perfect} for an example.

While perfect shuffles are the basis for many flashy card tricks today, Persi Diaconis (a professional magician and mathematician) and Ron Graham estimated in 2011 that it takes ``a few hundred hours, certainly thousands of repetitions" for a person to learn to do a perfect shuffle. They also added, ``We estimate that there are fewer than a hundred people in the world who can do eight perfect shuffles in under a minute."\cite{diaconis2011} Under the circumstances, it is natural to wonder if an easier shuffling technique could provide the same opportunities for impressive tricks. With this motivation, Doug Ensley introduced so-called {\em unshuffles} in his article ``Unshuffling for the imperfect magician."~\cite{Ensley}

\begin{figure}
\begin{center}
\includegraphics[width=11cm]{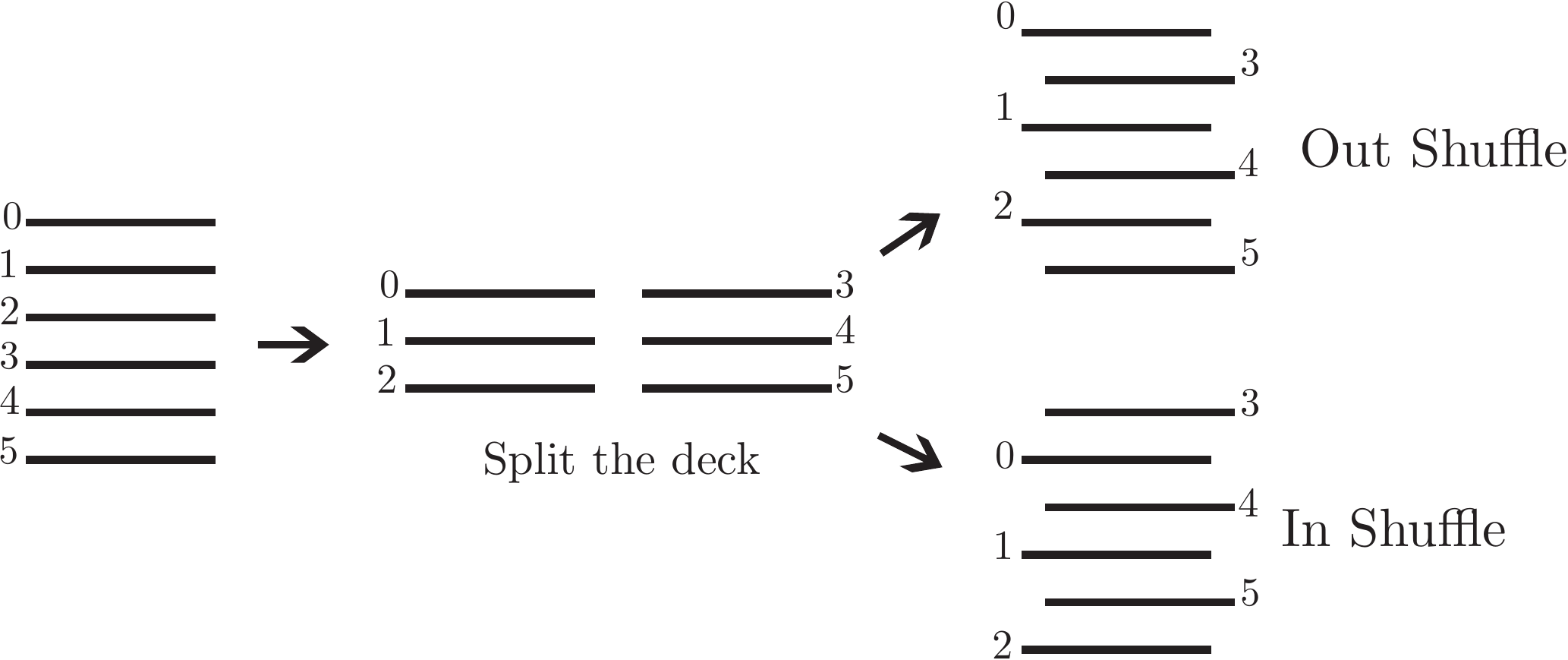} 
\end{center}
\caption{Perfect shuffles on a deck of 6 cards.}\label{perfect}
\end{figure}

Unshuffles are not nearly so difficult. Take a deck of $2n$ cards, and deal the cards from the top of the deck alternately into two piles, starting with one pile on the left and then one on the right. Reassemble the deck by stacking one pile on top of the other. There are two ways that this can be done. The {\em right shuffle} (denoted $R$) stacks the right pile on top of the left. The {\em left shuffle} ($L$) stacks the left pile on top of the right. After a right shuffle, notice that the top and bottom cards have now swapped places, while after a left shuffle, the original top and bottom cards are now the center two cards of the deck. See Figure~\ref{unshuffle} for an example with 6 cards. 

Ensley investigated unshuffles and described several tricks that can be performed via unshuffles by capitalizing on their mathematical properties~\cite{Ensley}. Our purpose here is to continue Ensley's investigation of unshuffles.

We have two main results. First, we consider Elmsley's Problem, a classic mathematical card trick where one moves the top card in the deck to any given position. We solve a generalization of Elmsley's Problem using unshuffles on a deck of $2^k$ cards (see Theorem~\ref{elmsley}).   For our second main result, we determine the structure of the permutation groups generated by left and right shuffles $\langle L, R \rangle$ on $2n$ cards. We prove that the group coincides with that of perfect shuffles for all $n$ such that  $n\not\equiv 3\pmod{4}$. If $n\equiv 3\pmod{4}$, then the group $\langle L, R \rangle$ with $2n$ cards is isomorphic to $B_n$, the group of all centrally symmetric permutations. Meanwhile, $\langle I, O \rangle$ is an index 2 subgroup of $B_n$ in this case. The full result is as follows.
\begin{theorem}\label{introtheorem}
Suppose a deck has $2n$ cards. 
\begin{enumerate}[(a)]
   \item If $2n = 12$, then $\langle L, R\rangle = \langle I, O\rangle$ which is isomorphic to the semi-direct product $\mathbb{Z}^6_2 \rtimes S_5$, where $S_5$ is the symmetric group on $5$ elements.
   \item If $2n=24$, then $\langle L, R\rangle=\langle I, O\rangle$ which is isomorphic to the semi-direct product $\mathbb{Z}^{11}_2 \rtimes M_{12}$, where the group $M_{12}$ is the Mathieu group of degree 12. The group has order $2^{11}\cdot 95040.$
   \item If $2n = 2^k$, then $\langle L, R\rangle=\langle I, O\rangle$ which is isomorphic to the semi-direct product $\mathbb{Z}^k_2 \rtimes \mathbb{Z}_k$.
\item If $n \equiv 0\pmod{4}$, $n > 12$, and $n$ is not a power of 2, then $\langle L, R\rangle = \langle I, O\rangle$ which is the intersection of the kernels of $sgn$ and $\overline{sgn}$ and has order $n!2^{n-2}$. (We define $sgn$ and $\overline{sgn}$ in Section~\ref{groups}.)
\item If $n \equiv 1 \pmod{4}$ and $n>1$, then $\langle L, R\rangle=\langle I, O\rangle$ which is the kernel $\overline{sgn}$ and has order $n!2^{n-1}$.
\item If $n \equiv 2\pmod{4}$ and $n > 6$, then $\langle L, R\rangle = \langle I, O\rangle=B_n$, where $B_n$ is the group of centrally symmetric permutations in $S_{2n}$ and has order $n!2^{n}$.
\item If $n \equiv 3\pmod{4}$, then $\langle L, R\rangle = B_n$ and has order $n!2^n$.

\end{enumerate}
\end{theorem}

From this result, we see that limiting oneself to unshuffles will not limit the card arrangements that one can reach as compared to using perfect shuffles. For $n\not\equiv 3 \pmod{4}$, unshuffles on a deck of $2n$ cards yield the exact same set of arrangements as perfect shuffles. When $n\equiv 3 \pmod{4}$, unshuffles allow a person to reach all the arrangements as perfect shuffles plus another set of arrangements of equal size which perfect shuffles do not reach. 

We begin in Section~\ref{properties} by investigating the mathematical properties of unshuffles. In Section~\ref{elmsleysection}, we provide a solution to a generalization of Elmsley's Problem on a deck of $2^k$ cards. We prove Theorem~\ref{introtheorem} in Section~\ref{groups} as follows. Theorem~\ref{special} proves parts (a) and (b) of Theorem~\ref{introtheorem}, Theorem~\ref{power} proves (c), and Theorem~\ref{main} proves (d)--(g). 

\begin{figure}
\begin{center}
\includegraphics[width=11cm]{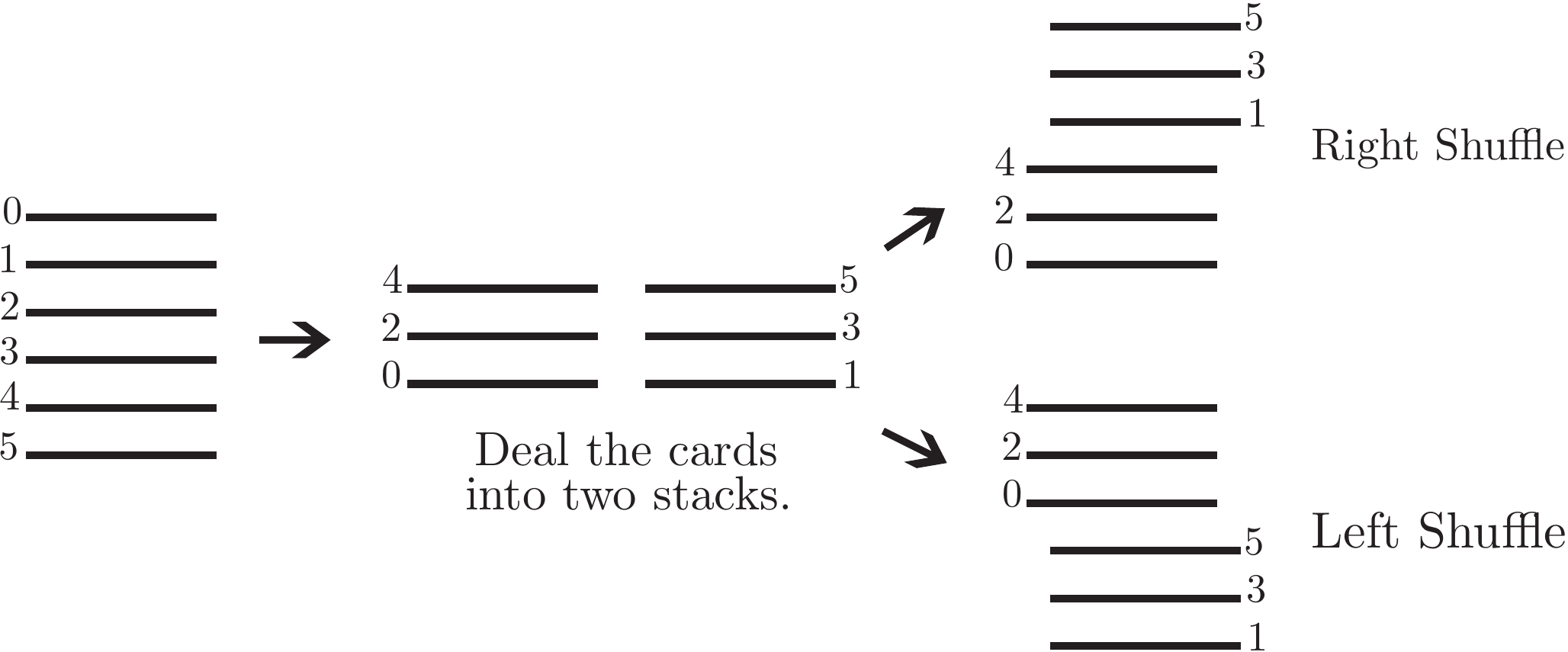} 
\end{center}
\caption{Unshuffles on a deck of 6 cards.}\label{unshuffle}
\end{figure}

\section{Properties of unshuffles}\label{properties}
Throughout our discussion, we consider decks with an even number of cards, denoted by $2n$. We index the cards by their distance from the top of the deck. So the top card has index 0, the second card from the top has index 1, and so on. The bottom card has index $2n-1$. We treat shuffles as functions and hence read their products from right to left.

A formula for the location of the $i^{th}$ card after an in or out perfect shuffle is known (see, for example,~\cite{Diaconis, Ensley}). These formulas can be stated in terms of modular arithmetic as follows:
$$I(i) = 2i + 1 \!\!\!\pmod{2n+1},\quad\text{ and } \quad 
O(i) = 
\begin{cases}
        2i  \!\!\! \pmod{2n-1} & \text{if } i\neq 2n-1\\
        2n-1 & \text{if } i =2n-1.
\end{cases} 
    $$

In the same way, we can find formulas to express the action of the left and right shuffles on a deck of cards. Notice that the formulas have similarities with those given above for in and out shuffles. 
    
\begin{lemma}\label{formula}
Suppose a deck has $2n$ cards. The index of the $i^{th}$ card after a left or right shuffle is given by the following formulas:
$$ L(i) = ni + n-1\!\!\!\! \pmod{2n+1},\quad\text{ and } \quad 
 R(i) = 
\begin{cases}
        (n-1)i \!\!\!\! \pmod{2n-1} & \text{if } i\neq 0\\
        2n-1 & \text{if } i =0. 
\end{cases}  
    $$    
\end{lemma}
\begin{proof}
    After we deal the deck of cards from the top of the deck alternately into two piles starting with the first card on the left, the cards are in the following array:
\begin{align*}~\label{array}
\begin{array}{cccc} 
2n-2 &&& 2n-1 \\ 
2n-4 &&& 2n-3\\ 
\vdots  &&& \vdots\\
2  &&& 3\\
0 && & 1
\end{array}
\end{align*}
Stacking the piles together with a left shuffle puts the cards in the following order: 
$$2n-2, 2n-4, \ldots, 2, 0, 2n-1, 2n-3, \ldots, 3, 1.$$
The index of the $i^{th}$ card after a left shuffle is then given by the equation:
\begin{equation}\label{L}
L(i) = 
\begin{cases}
        -\tfrac{1}{2}i + n-1 & \text{if $i$ is even }\\
        -\tfrac{1}{2}(i+1) + 2n & \text{if $i$ is odd}
\end{cases}
\end{equation}  
We want to write $L(i)$ in terms of an expression in $\mathbb{Z}_{2n+1}$. Observe that in $\mathbb{Z}_{2n+1}^*$, we have $(-2)^{-1} \equiv n \pmod{2n+1}$ and also $2n \equiv -1\pmod{2n+1}$. Therefore, no matter whether $i$ is even or odd, we can express $L(i)$ as follows: $L(i)= ni + (n-1) \pmod{2n+1}.$

Now let's consider right shuffles. After stacking the cards with the right stack on top, the cards are in the following order:
$$2n-1, 2n-3, \ldots, 3, 1,2n-2, 2n-4, \ldots, 2, 0.$$
 So then for even $i$, the index $R(i)$ is given simply by adding $n$ to $L(i)$. For odd $i$, $R(i)$ is given by subtracting $n$ from $L(i)$. Thus:
      $$ R(i) = 
\begin{cases}
        -\tfrac{1}{2}i + 2n-1 & \text{if $i$ is even }\\
        -\tfrac{1}{2}(i+1) + n & \text{if $i$ is odd}
\end{cases}
    $$  
We want to write $R(i)$ in terms of an expression in $\mathbb{Z}_{2n-1}$. This cannot be done for $i=0$, since $R(0) = 2n-1$, so that is a special case. But for $i>0$, we proceed as follows. Observe that in $\mathbb{Z}_{2n-1}^*$, we have $(-2)^{-1} \equiv n-1 \pmod{2n+1}$ and also $n \equiv 1-n \pmod{2n-1}$. Making the appropriate substitutions, we find that for all $i>0$, we can express the above formula for $R$ as:
$R(i)= (n-1)i \pmod{2n-1}.$
\end{proof}
We note that a right shuffle swaps the places of the outermost cards (cards $0$ and $2n-1$) and performs a left shuffle on the inner $2n-2$ cards. See Figure~\ref{unshuffle} for an example.

Using the above formulas, we can now determine how many times one must do a left (or right) shuffle so that the deck returns to its original arrangement. In mathematical terms, we are finding the {\em order} of the shuffle. 

\begin{proposition}
    Suppose a deck has $2n$ cards. The order of the left shuffle is the order of $-2$ in $\mathbb{Z}_{2n+1}^{*}$. Denote the order of $-2$ in $\mathbb{Z}_{2n-1}^{*}$ by $r$. If $r$ is even, then the right shuffle has order $r$. If $r$ is odd, right shuffle has order $2r$.

\end{proposition}
\begin{proof}

If we repeat the left shuffle $k$ times we have:
$L^k(i) = n^ki + n^k - 1\pmod{2n+1}$.
If $k$ is the order of $L$, then $n^ki + n^k - 1 \equiv i \pmod{2n+1}$ for all $i$. This implies $n^k \equiv 1\pmod{2n+1}$. In other words, the order of $L$ is the order of $n$ in $\mathbb{Z}_{2n+1}^{*}$. Equivalently, the order of $L$ is the order of the inverse of $n$ in $\mathbb{Z}_{2n+1}^{*}$, which is $-2$.

    Next consider the right shuffle. This takes slightly more work because we must consider the inner cards of the deck separately from the top and bottom cards (cards 0 and $2n-1$). Let's begin with the cards with indices $0<i<2n-1$. The right shuffle formula for these cards is $R(i) = (n-1)i$, which is equivalent to $R(i) = -ni$, because $n-1 \equiv -n\pmod{2n-1}$. Composing $R$ with itself $k$ times, we have $R^k(i) = (-n)^ki$ for $0<i<2n-1$. We have $R^k(i) \equiv (-n)^ki \equiv i\pmod{2n-1}$ for all $0<i<2n-1$ if and only if $k$ is a multiple of the order of $-n$. Equivalently, $k$ is a multiple of the order of $-2$ in $\mathbb{Z}_{2n-1}^{*}$ (which is the inverse of $-n$ in $\mathbb{Z}_{2n-1}^{*}$). 
    
     Now we turn to the top and bottom cards of the deck (cards 0 and $2n-1$). Since the top and bottom cards swap places after one right shuffle, $k$ must be an even integer in order to have $R^k(0)=0$ and $R^k(2n-1) = 2n-1$. It follows that the order of $R$ is the order of $-2$ in $\mathbb{Z}_{2n-1}^{*}$ if that value is even. Otherwise, the order of $R$ is twice the order of $-2$ in $\mathbb{Z}_{2n-1}^{*}$.
\end{proof}
\begin{example}
Suppose we have 52 cards.  The order of the left shuffle will be the order of $-2$ in $\mathbb{Z}_{53}^{*}$, which is $52$. The order of $-2$ in $\mathbb{Z}_{50}^{*}$ is 8, so the order of the right shuffle is 8. 
\end{example}

\begin{figure}
\begin{center}
\includegraphics[width=5cm]{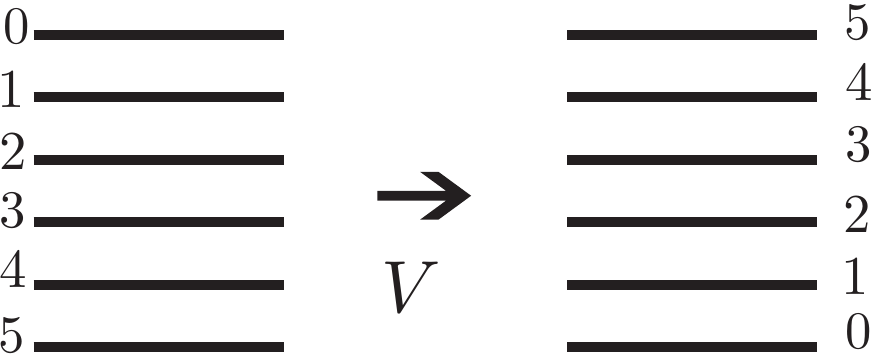} 
\end{center}
\caption{The shuffle $V$ on a deck of 6 cards.}\label{V}
\end{figure}

We now introduce a third type of shuffle, denoted by $V$. This shuffle simply reverses the order of the cards. See Figure~\ref{V} for an example. The shuffle $V$ changes the index of the $i^{th}$ card as follows:
$$V(i) = 2n-1-i.$$
Observe that $V$ has order 2. This new shuffle $V$ helps us define the connection between left and right shuffles and perfect shuffles $I$ and $O$. This connection is discussed in Proposition 2 of~\cite{Ensley}, as well. 

\begin{proposition}\label{connection}
Suppose a deck has 2n cards. Left and right shuffles are related to in and out shuffles and $V$ as follows:
$$L = V I^{-1} = I^{-1} V $$
$$R = V O^{-1} = O^{-1} V$$
\end{proposition}
\begin{proof}
    We first prove that $L I = I L$. Using the formulas for the left shuffle and the in shuffle, we compute $L(I(i)) = n(2i+1) + n - 1 = 2ni + 2n - 1 \pmod{2n+1}$ and $I(L(i)) = 2(ni + n - 1) + 1 = 2ni + 2n - 1 \pmod{2n+1}$. The expressions for $LI$ and $IL$ are equal, as desired. Next, observe that: $$L(I(i))=2ni + 2n - 1 \equiv 2ni + i + 2n - 1 - i \equiv 2n - 1 - i =V(i)\pmod{2n+1}.$$ 
    Thus, we have proved that $L I  = I L = V$. Equivalently, $L = V I^{-1} = I^{-1} V. $ 
    
    Using a similar process, we can also prove  $R = V O^{-1} = O^{-1} V.$ We leave this to the reader.
\end{proof}

Observe from the above lemma that a left shuffle is ``close" to being the inverse of an in shuffle (up to reversing the order of the deck), and the right shuffle is similarly close to being the inverse of an out shuffle. 

In Section~\ref{groups}, we will be interested in writing the shuffle $V$ either as a product of perfect shuffles or as a product of unshuffles. The following lemma describes a special case where this can be done easily, which will be useful to us in Theorem~\ref{special}.

\begin{lemma}\label{relation}
If $V = I^y$ where $y$ is even, then $V = L^y$.
\end{lemma}
\begin{proof}
By Proposition~\ref{connection}, $L = I^{-1}V=VI^{-1}$. Using this and the facts that $V = I^y$, $y$ is even, and $V^2$ is the identity, we compute:
$L^y = (I^{-1}V)^y = (I^y)^{-1}V^y =V. $ 
\end{proof}

A valuable property of perfect shuffles goes as follows. Take two cards that are equidistant from the center of the deck. After an in or out shuffle, the two cards remain equidistant from the center of the deck. We make this precise in the following definition.

\begin{definition}
A permutation $\sigma$ of $\{0,1,\ldots, 2n-1\}$ {\bf preserves central symmetry} if for every $i, j\in \{0,1,\ldots, 2n-1\}$ such that $i + j = 2n-1$, we have $\sigma(i) + \sigma(j) = 2n-1$.
\end{definition}

Among magicians, this principle of preserving central symmetry is termed {\em stay-stack}. The fact that perfect shuffles preserve central symmetry was first pointed out in 1957 by Russell Duck, a Pennsylvania policeman (see Section 3 of \cite{Diaconis}), and the observation can be exploited to create all sorts of card tricks. It is straightforward to show that unshuffles have this property, as well.

\begin{proposition}
Left and right shuffles preserve central symmetry.
\end{proposition}
\begin{proof}
Perfect shuffles $I$ and $O$ preserve central symmetry. Hence so do their inverses $I^{-1}$ and $O^{-1}$. The shuffle $V$ has the effect of swapping the cards in each centrally symmetric pair, and so $V$ also preserves central symmetry. Left and right shuffles are compositions of $V, I^{-1}$, and $O^{-1}$ (Proposition~\ref{connection}), so $L$ and $R$ must also preserve central symmetry.
\end{proof}

\section{Elmsley's Problem with unshuffles}\label{elmsleysection}

In the 1950's, Scottish computer programmer and magician Alex Elmsley posed a problem which grew popular and has since been given his name:
\begin{quote}{\bf Elmsley's Problem.}
Is it possible to move the top card in the deck to any given position via perfect shuffles? 
\end{quote}
Elmsley himself found a clever solution, which he published in the card magazine {\em Ibidem} (No. 11, September 1957, see also~\cite{diaconis2011, Diaconis, Morris}). The solution is as follows.

\begin{theorem}[Solution to Elmsley's Problem]
Begin with a deck of $2n$ cards. To move the top card to position $i$, express $i$ in binary, and then, reading the binary expression from left to right, let 1 represent an in shuffle and let 0 represent an out shuffle. Performing this indicated sequence of in and out shuffles in order from left to right will move card $0$ to position $i$.
\end{theorem}

This slick solution to Elmsley's Problem encouraged quite a bit of related work. Paul Swinford discovered that in the special case where the deck has $2^k$ cards, the sequence of shuffles described in the above proposition swaps the places of the top card and the $i^{th}$ card. More generally, he observed that a sequence of shuffles that brings card $i$ to position $j$ will also move card $j$ to position $i$~\cite{swinford1,swinford2}. Later, Ramnath and Scully described a way to move card $i$ to card $j$ using perfect shuffles for any deck of $2n$ cards~\cite{Ramnath}. Diaconis and Graham found a solution to the inverse of Elmsley's Problem (that is, bringing card $i$ to the top of the deck via perfect shuffles)~\cite{DiaconisGraham}. Elmsley's Problem has also been solved for so-called generalized perfect shuffles~\cite{Medvedoff}. 

Performing this trick with perfect shuffles is an impressive feat.\footnote{See the trick performed and discussed here: \url{https://youtu.be/Y2lXsxmBx7E}} Having the option to use unshuffles will make it easier. We give a solution to Elmsley's Problem for unshuffles in the special case where the deck has $2^k$ cards. We first prove a lemma determines the location of the $i^{th}$ card after a left or right shuffle in terms of the binary representation of $i$.

\begin{lemma}\label{binary}
    Suppose we have a deck of $2^k$ cards for some $k \geq 1$. Write $i$ where $0\leq i\leq 2n-1$ in binary notation: $i = x_{k-1}~x_{k-2}\cdots x_1~x_0$, and let  $\overline{x}_j = 1-x_j$. Using the binary expansion, left and right shuffles move the $i^{th}$ card to a new position as follows:

    \vspace{-1cm}
    \begin{align*}
        &L(x_{k-1}~x_{k-2}\cdots x_1~x_0) = x_{0}~\overline{x}_{k-1}~\overline{x}_{k-2} \cdots \overline{x}_{2}~\overline{x}_{1}, \\
        &R(x_{k-1}~x_{k-2}\cdots x_1~x_0)= \overline{x}_{0}~\overline{x}_{k-1}~\overline{x}_{k-2} \cdots \overline{x}_{2}~\overline{x}_{1}.
    \end{align*}  
\end{lemma}

\begin{proof}
    Using our left shuffle formula (Lemma~\ref{formula}) and the fact that $2n = 2^k$, we have $L(i) = 2^{k-1}i + (2^{k-1}-1) \pmod {2^k+1}.$ 
    
    Plug in the value $i=x_{k-1}\cdot2^{k-1} + x_{k-2}\cdot2^{k-2} + \cdots + x_{0}\cdot2^{0}$: 
    $$L(i) = x_{k-1}\cdot2^{2k-2} + x_{k-2}\cdot2^{2k-3} + \cdots + x_{0}\cdot2^{k-1} + 2^{k-1} - 1 \pmod {2^k+1}.$$ 
The above expression can be simplified in two ways. First note that we can substitute $2^{k-1}-1 = 1\cdot2^{k-2}+1\cdot2^{k-3}+\cdots+1\cdot2^{0}$. Secondly, since $2^k \equiv -1\pmod{2^k+1}$, we can replace $x_{j}\cdot2^{j+(k-1)}$ with $-x_{j}\cdot2^{j-1}$. 
\begin{align*}
    L(i)&\equiv -x_{k-1}\cdot2^{k-2}-x_{k-2}\cdot2^{k-3}-\cdots-x_{1}\cdot2^{0}+x_{0}\cdot2^{k-1}+2^{k-1}-1\\ 
    &\equiv x_{0}\cdot2^{k-1}+ (1-x_{k-1})\cdot2^{k-2}+(1-x_{k-2})\cdot2^{k-3}+\cdots +(1-x_{1})2^{0}\\
    &\equiv x_{0}\cdot2^{k-1}+ \overline{x}_{k-1}\cdot2^{k-2}+\overline{x}_{k-2}\cdot2^{k-3}+\cdots +\overline{x}_{1}2^{0} \pmod{2^k+1}
\end{align*}
 where $\overline{x}_j = 1-x_j$.   
We have reached our desired conclusion: $$L(x_{k-1}~x_{k-2}\cdots x_1~x_0) = x_{0}~\overline{x}_{k-1}~\overline{x}_{k-2} \cdots \overline{x}_{2}~\overline{x}_{1}.$$ 

We use the same method for right shuffles. The right shuffle formula for $i>0$ is: 
$R(i) = (2^{k-1}-1)i \pmod {2^k-1}$.
Since $2^{k-1}-1 \equiv -2^{k-1} \pmod{2^k-1},$ we have
$$R(i) \equiv -2^{k-1}i \pmod {2^k-1}$$
for $i>0.$
Plugging in $i =x_{k-1}\cdot2^{k-1} + x_{k-2}\cdot2^{k-2} + \cdots + x_{0}\cdot2^{0}$ and using the fact that $2^k \equiv 1\pmod{2^k-1}$, we have:
\begin{align*}
   R(i) &\equiv (-2^{k-1})(x_{k-1}\cdot2^{k-1} + x_{k-2}\cdot2^{k-2} + \cdots +x_1\cdot2^1 + x_{0}\cdot2^{0})\\
   &\equiv -x_{k-1}\cdot2^{2k-2} -x_{k-2}\cdot2^{2k-3} - \cdots - x_1\cdot 2^{k}-x_{0}\cdot2^{k-1} \\
   &\equiv -x_{k-1}\cdot2^{k-2} -x_{k-2}\cdot2^{k-3} -\cdots -x_{1}\cdot2^{0}-x_{0}\cdot2^{k-1} \pmod{2^k-1}
\end{align*}
Now note that $1\cdot2^{k-1}+1\cdot2^{k-2}+\cdots+1\cdot2^{0}=2^{k}-1 \equiv 0\pmod{2^k-1}$. So we can add this to the above expression without changing its value:
\begin{align*}
R(i) &\equiv (1-x_{0})\cdot2^{k-1} +(1-x_{k-1})\cdot2^{k-2}+\cdots +(1-x_{1})\cdot2^{0} \\
&\equiv \overline{x}_{0}\cdot2^{k-1} +\overline{x}_{k-1}\cdot2^{k-2}+\cdots +\overline{x}_{1}\cdot2^{0} \pmod{2^k-1},
\end{align*}
 where $\overline{x}_j = 1-x_j$. This proves our result for $i>0$. 
 
 For the special case $i=0$, observe that we have the following, as desired: $$R(0) = 2^k-1 = 1 \cdot 2^{k-1} + 1 \cdot 2^{k-2} + \cdots + 1 \cdot 2^{0}.$$
\end{proof}
\begin{figure}
\begin{center}
\includegraphics[width=11cm]{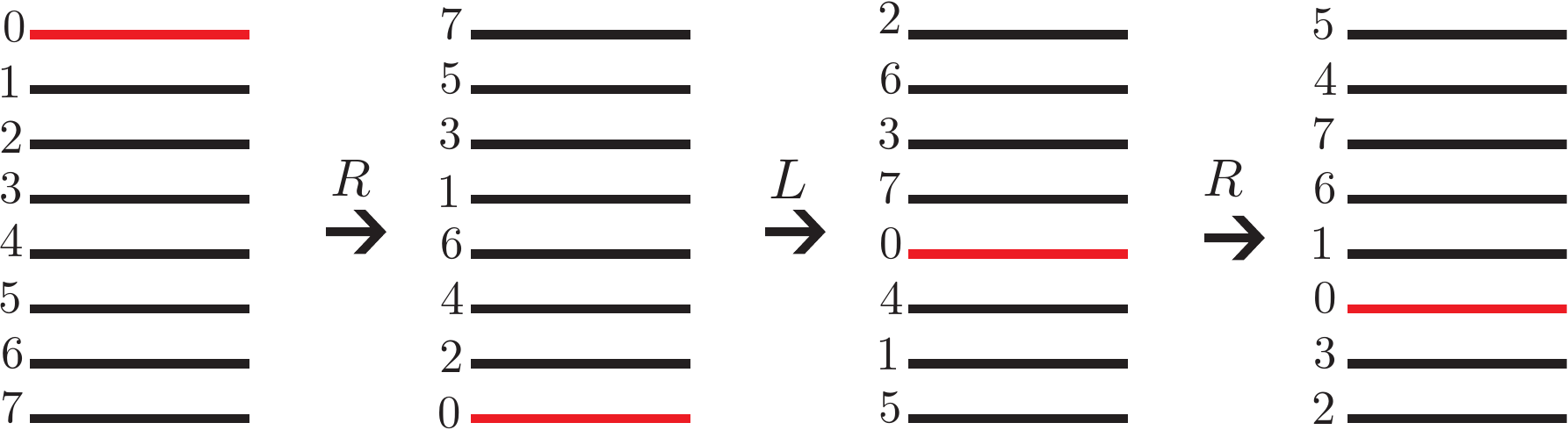} 
\end{center}
\caption{Unshuffles that move the top card to position 5.}\label{Elmsely8}
\end{figure}

We now give and prove our solution to Elmsley's Problem using unshuffles in the case that there are $2^k$ cards. In fact, we solve a more general problem. The process that we describe {\em swaps the places of card $i$ and card $j$ in $k$ shuffles for any $i$ and $j$}. Thus, we can specialize to the case $i=0$ to recover the solution to Elmsley's Problem.

\begin{theorem}\label{elmsley}
Suppose there are $2^k$ cards for some $k\geq 1$. To swap the cards in positions $i$ and $j$ using unshuffles, write $i$ and $j$ in binary, and compute $i\oplus j = x_{k-1}~x_{k-2}\cdots x_1~x_0$ (where $\oplus$ denotes the xor operation). Shuffle the deck with $k$ shuffles denoted in order by $S_0, S_1, \ldots, S_{k-1}$, where $S_r$ is defined as follows:
\begin{enumerate}
\item If $k$ is odd, 
$$ S_r = 
\begin{cases}
        L & \text{if } x_r=0\\
        R & \text{if } x_r =1 
\end{cases}  
    $$
\item If $k$ is even, 
$$ S_r = 
\begin{cases}
        R & \text{if } x_r=0\\
        L & \text{if } x_r =1 
\end{cases} 
    $$
\end{enumerate}
This sequence of shuffles will swap the cards in positions $i$ and $j$.
\end{theorem}
\begin{proof}
Lemma~\ref{binary} tells us that both a left shuffle and right shuffle have the effect of sliding the bits of the card index to the right by 1 position and moving the last bit to the front. Since there are $k$ bits total, it follows that performing $k$ shuffles (left or right or a mixture of both) will return all bits back to their original position (with some of the bits flipped, perhaps). Moreover, after these $k$ shuffles are done, each bit will have flipped a total of either $k$ or $k-1$ times. The exact number of times a given bit $x_r$ is flipped is pinned down by what type of shuffle the $r^{th}$ shuffle was. The bit $x_r$ will have flipped $k$ times if $S_r$ is a right shuffle, and $x_r$ will have flipped $k-1$ times if $S_r$ is a left shuffle. 

Using this we can design a sequence of $k$ shuffles $S_0, S_1, \ldots, S_{k-1}$ that turns the binary expansion for $j$ into that of $i$ and vice versa. First, identify the bits where $j$ and $i$ differ via computing: $i\oplus j= x_{k-1}~x_{k-2}\cdots x_1~x_0$. 

If $x_r = 0$, then we need no change to the $r^{th}$ bit of $j$, and we can plan our shuffles according to the parity of $k$ so that the $r^{th}$ bit is unchanged when the shuffles are finished.  Namely, if $k$ is odd, then set $S_r = L$ because that will cause the $r^{th}$ bit of $j$ to flip $k-1$ times and hence remain unchanged. If $k$ is even, set $S_r = R$, and the $r^{th}$ bit of $j$ will flip $k$ times and hence remain unchanged. 

Similarly, if $x_r = 1$, then the bits of $i$ and $j$ differ. We plan our shuffles according to the parity of $k$ so that the $r^{th}$ bit is flipped. If $k$ is odd, then set $S_r = R$, and the $r^{th}$ bit of $j$ will flip. If $k$ is even, set $S_r = L$.

This defines a sequence of $k$ shuffles $S_0, S_1, \ldots, S_{k-1}$. Performing these shuffles in order will change all the bits of $j$ to coincide with the bits of $i$ and vice versa. Therefore, the cards in positions $i$ and $j$ will trade positions after the sequence of shuffles.
\end{proof}

\begin{example}
We illustrate this solution to Elmsley's Problem with two examples, one with $k$ odd and the other with $k$ even. Suppose there are 8 cards and we want the top card to trade positions with the $5^{th}$ card. We compute $000_2 \oplus 101_2 = 101_2$. The shuffles should be done in the order of $R, L, R$. See Figure~\ref{Elmsely8}.

Now suppose there are 16 cards and we want to swap the cards in the $6^{th}$ position and the $11^{th}$ position. In binary, we write $6 = 0110_2$ and $11=1011_2$. We compute $0110_2 \oplus 1011_2 = 1101_2$. Theorem~\ref{elmsley} tells us the shuffles $L, R, L, L$ in that order will swap these two cards. See Figure~\ref{Elmsely16}.
\end{example}

\begin{figure}
\begin{center}
\includegraphics[width=11cm]{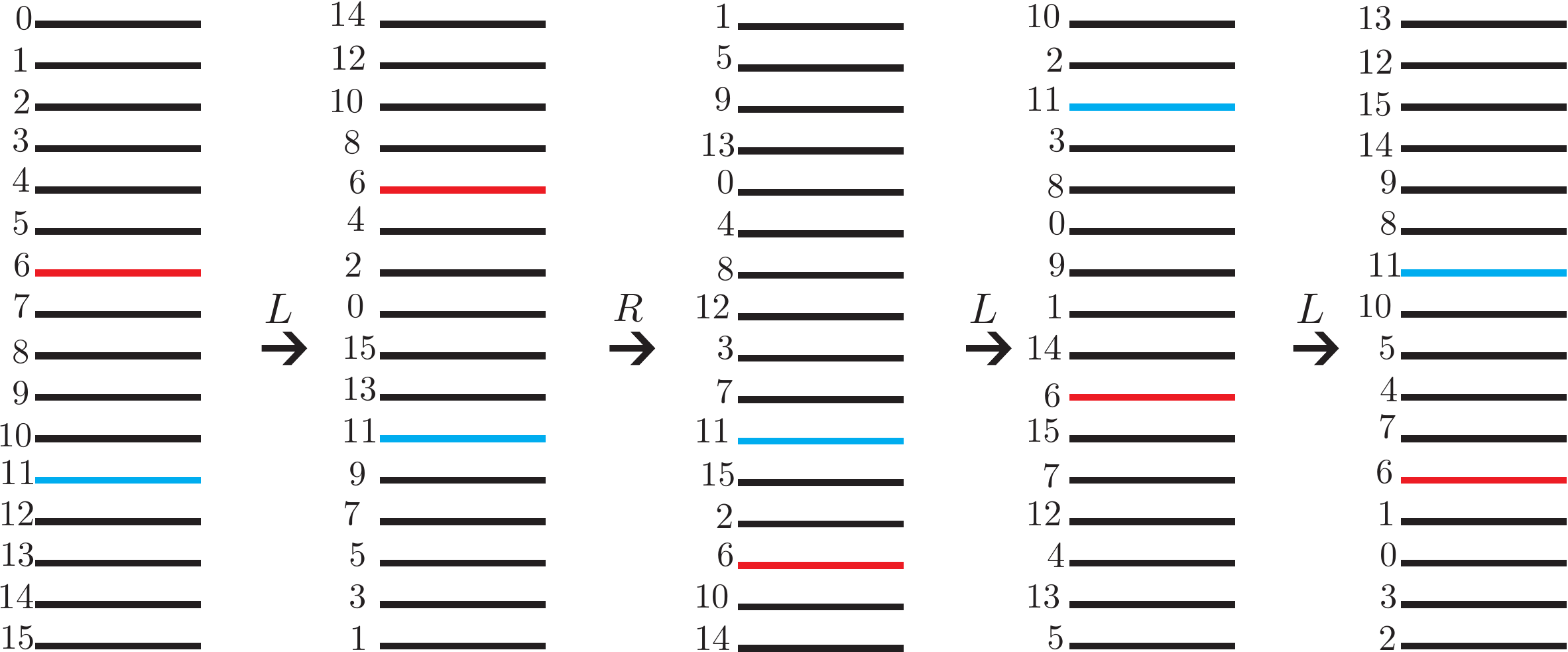} 
\end{center}
\caption{Unshuffles that swap the positions of card 6 and card 11.}\label{Elmsely16}
\end{figure}


\section{The permutation groups of unshuffles}\label{groups}

We now find structure of the permutation group $G = \langle L, R\rangle$. Through this, we will know exactly what card arrangements can be reached using unshuffles. In some (but not all) cases, the permutation group $\langle L, R\rangle$ coincides with $\langle I, O\rangle$. The fact that the groups often coincide is not so surprising, given the close relationship between perfect shuffles and unshuffles (Proposition~\ref{connection}). However, the details require careful work. We prove Theorem~\ref{introtheorem} in a sequence of three theorems (Theorems~\ref{special},~\ref{power}, and~\ref{main}). 

Let's begin by reviewing the permutation groups of perfect shuffles. Let $B_n$ be the group of all centrally symmetric permutations of $2n$ elements. This group can also be identified as the group of all signed $n\times n$ permutation matrices. Because $L, R, I$, and $O$ are all themselves centrally symmetric permutations, it follows that both $\langle L, R\rangle$ and $\langle I, O\rangle$ are subgroups of $B_n$.

We define several homomorphisms on $B_n$. First let
$$sgn: B_n\longrightarrow \{\pm 1\}$$
be the homomorphism such that $sgn(g)$ is the sign (or, parity) of the permutation $g$. Next, observe that while each permutation $g \in B_n$ is a permutation on $2n$ elements, $g$ induces a permutation on the $n$ centrally symmetric pairs. We define a homomorphism 
$$\phi:B_n \longrightarrow S_n$$ 
by assigning $\phi(g)$ to be the permutation that $g$ induces on the $n$ centrally symmetric pairs. Next we define 
$$\overline{sgn}: B_n\longrightarrow \{\pm 1\}$$
by assigning $\overline{sgn}(g)$ to be the sign of the permutation $\phi(g)\in S_n$. We easily pick up a final homomorphism as follows. Let $sgn\overline{sgn}: B_n\longrightarrow \{\pm 1\}$ be the homomorphism that assigns $g$ to the product $sgn(g)\overline{sgn}(g)$. 

Now we are ready to state the group structure of $\langle I, O\rangle$. The first three items are special cases (when the deck has size 12, 24, or a power of 2, respectively). Beyond that, the group structure is determined by the congruence class of $n$ modulo 4. For a proof and full details, see~\cite{Diaconis}.

\begin{theorem}\cite{Diaconis}\label{perfectgroups}
    The structure of the permutation group $\langle I, O\rangle$ on $2n$ cards is as follows:
    \begin{enumerate}
     \item If $2n = 12$, then $\langle I, O\rangle$ is isomorphic to the semi-direct product $\mathbb{Z}^6_2 \rtimes S_5$, where $S_5$ is the symmetric group on $5$ elements.
   \item If $2n=24$, then $\langle I, O\rangle$ is isomorphic to the semi-direct product $\mathbb{Z}^{11}_2 \rtimes M_{12}$, where $M_{12}$ is the Mathieu group of degree 12.
   \item If $2n = 2^k$, $\langle I, O\rangle$ is isomorphic to the semi-direct product $\mathbb{Z}^k_2 \rtimes \mathbb{Z}_k$.
  
   \item If $n \equiv 0 \pmod{4}$, $n> 12$, and $n$ is not a power of 2, then $\langle I, O\rangle$ is the intersection of the kernels of $sgn$ and $\overline{sgn}$ and has order $n!2^{n-2}$.
   
\item If $n \equiv 1 \pmod{4}$ and $n>1$, then $\langle I, O\rangle$ is the kernel $\overline{sgn}$ and has order $n!2^{n-1}$.

\item If $n \equiv 2 \pmod{4}$ and $n>6$, then $\langle I, O\rangle = B_n$ and has order $n!2^{n}$.
   
   \item If $n \equiv 3 \pmod{4}$, then $\langle I, O\rangle$ is equal to the kernel of $sgn\cdot\overline{sgn}$ and has order $n!2^{n-1}$.

    \end{enumerate}
\end{theorem}

Now we turn to consider the group $G = \langle L,R \rangle$. We begin by proving $G = \langle L,R \rangle$ coincides with the perfect shuffle group in the cases where $2n$ is 12 or 24 (Theorem~\ref{special}) and in the case where $2n$ is a power of 2 (Theorem~\ref{power}). 
\newpage
\begin{theorem}\label{special}
   If $2n = 12$ or $2n = 24$, then $\langle L, R\rangle = \langle I, O\rangle$.
\end{theorem}
\begin{proof}
   We first observe that the formula for $I^r$ is give by:
$$I^r(i) \equiv 2^ri + 2^{r-1} + 2^{r-2} + ... + 1 \equiv ~~2^ri + 2^r - 1 \pmod{2n+1}.$$
Now suppose that we have $2n=12$ cards. Observe that $$I^6(i) = 2^6(i) + 2^6-1 \equiv 11-i =V(i)\pmod{13}.$$ 
Therefore $V = I^6$. Lemma~\ref{relation} now implies that $V = L^6$. Because we can express $V$ in terms of $I$, Proposition~\ref{connection} implies that $\langle L, R\rangle \subseteq \langle I, O\rangle$. Moreover, because we can express $V$ in terms of $L$, it implies $\langle I,O\rangle \subseteq \langle L, R\rangle$. Therefore $\langle L, R\rangle = \langle I, O\rangle$.

   Similarly, for $2n=24$ cards we have
   $$I^{10}(i) = 2^{10}(i) + 2^{10}-1 \equiv 23-i = V(i)\pmod{25}.$$ Because $V = I^{10}$, Lemma~\ref{relation} implies $V = L^{10}$. The same argument for $12$ cards applies to $24$ cards, so $\langle L, R\rangle = \langle I, O\rangle$ for the $2n = 24$ case, as well.   
\end{proof}

\begin{theorem}\label{power}
Suppose $2n=2^k$ for some positive integer $k$. Then $\langle L,R \rangle = \langle I, O\rangle$.
\end{theorem}
\begin{proof}

Using the formula for $I^r(i)$ mentioned in the proof of Theorem~\ref{special} and setting $r=k$, we observe:
$$I^k(i) \equiv 2^ki + 2^k - 1 \equiv - i + 2^k - 1 =V(i) \pmod{2^k+1}.$$
Therefore $V = I^k.$ This implies that $L$ and $R$ can be written as combinations of $I$ and $O$ by Proposition~\ref{connection}. So then $\langle L, R\rangle \subseteq \langle I, O\rangle.$

Now we prove $\langle I, O\rangle \subseteq \langle L, R\rangle$, splitting into two cases according to the parity of $k$. Suppose that $k$ is even. Begin with the equation $L = VI^{-1}$ from Proposition~\ref{connection}. Composing both sides with themselves $k$ times, we have
$$L^k = (VI^{-1})(VI^{-1})\cdots(VI^{-1}) = V^k I^{-k} = I^{-k},$$ because $V$ and $I^{-1}$ commute and $k$ is even.  Plugging in $V=I^k$ and remembering that $V=V^{-1}$, we find that $L^k = V.$ This implies that $I$ and $O$ can be written in terms of $L$ and $R$ for $k$ even. Therefore $\langle I, O\rangle \subseteq \langle L, R\rangle$, which proves $\langle I, O\rangle = \langle L, R\rangle$ for $k$ even, as desired

Suppose $k$ is odd. Begin with $R = V O^{-1}$ (from Proposition~\ref{connection}) and compose both sides with themselves $k$ times. Remembering that $O^{-1}$ and $V$ commute and using the fact that $O^{-k}$ is the identity (because $O^k(i) = 2^ki \equiv i \pmod{2^k-1}$), we find
$$R^k = (VO^{-1})(VO^{-1})\cdots(VO^{-1}) = V^k O^{-k} = V^k = V.$$ 
Therefore when $k$ is odd, $V = R^k.$
Similar to before, we conclude $\langle I, O\rangle \subseteq \langle L, R\rangle$, which proves $\langle I, O\rangle = \langle L, R\rangle$ for $k$ odd, as desired
\end{proof}

Now that we have settled the above special cases, we have the more difficult task of considering what happens in general. We must establish some preliminary definitions and two lemmas first. 

Let $G^*$ denote the subgroup of $G=\langle L,R\rangle$ consisting of all shuffles in $G$ that leave the set $\{0, 1, \ldots, n-1\}$ invariant. Since all shuffles in $G^*$ also must preserve central symmetry, it follows that elements in $G^*$ can be expressed as $\sigma\sigma'$ where $\sigma$ is a permutation of $\{0, 1, \ldots, n-1\}$ and $\sigma'$ represents the corresponding permutation of the elements $\{0', 1', \ldots, (n-1)'\}$ where $r' = 2n-1-r$. So an element $\sigma\sigma'$ in $G^*$ is completely determined by $\sigma$. Using this notation throughout, we now prove two lemmas. 

\begin{lemma}\label{G*}
    Let $n> 1$ be such that $n$ is not a power of 2 and $n\neq 6, 12$. The group $G^*$ contains all permutations of the form $\sigma\sigma'$ where $\sigma$ is any even permutation of $\{0, 1, \ldots, n-1\}$ and $\sigma'$ is as defined above.
\end{lemma}
\begin{proof}
In~\cite{Diaconis}, the authors introduce several useful permutations which are created with perfect shuffles and will be useful in our context, as well. Let $k$ and $r$ be positive integers and define:
\begin{align*}
c &= O(I^{-1}OIO^{-1})^2O^{-1}\\
w &=  O^{-1}Ic^{-1}O^{-1}Ic^2I^{-1}Oc^{-1}I^{-1}O\\
b &= (I^kO^{-k}I^{-1}O)^{-2}\\
c' &= ObO^{-1}\\
h(r) &= O^{-r}I^r 
\end{align*}

All of these permutations can be performed using left and right shuffles, as well. To see this, substitute $I = VL^{-1}$ and $O = VR^{-1}$ into the above expressions. Because $V$ commutes with both $L$ and $R$ and because $V^2$ is the identity, we are left with a product of left and right shuffles. 
    
 In Lemma 9 of~\cite{Diaconis}, the authors prove that when $n$ is odd, $c$ and $w$ generate all permutations $\sigma\sigma'$ where $\sigma$ is an even permutation of $\{0, 1, \ldots, n-1\}$.  Therefore our result is proved for $n$ odd.

Now suppose $n$ is even, $n$ is not a power of 2, and $n\neq 6,12$. Write $2n = 2^kv$ where $v>1$ and $v$ is odd. 
In Lemmas 15, 16, 17, and 18 of~\cite{Diaconis}, the authors prove that under the stated conditions on $n$, the shuffles $b, c', h(1), h(2), \ldots, h(k-1)$ generate all permutations $\sigma\sigma'$ where $\sigma$ is an even permutation of $\{0, 1, \ldots, n-1\}$. Therefore our result for left and right shuffles follows, as well.
\end{proof}

We must prove one more preliminary lemma, but first we establish some notation and make observations. Recall the homomorphism $\phi:B_n\longrightarrow S_n$ is defined by assigning $\phi(g)$ to be the permutation that $g$ induces on the $n$ centrally symmetric pairs. In~\cite{Diaconis}, the authors found that the parities of the permutations $I$, $O$, $\phi(I),$ and $\phi(O)$ depended on the congruence class of $n$ modulo $4$ (see Table 3 in~\cite{Diaconis}). Also remember that $V$ switches the card in position $i$ with the card in position $i'$ for $0 \leq i \leq n-1$, so $V$ is the product of $n$ transpositions:
$V = (0, 0')(1, 1')(2,2')\cdots(n-1,(n-1)').$ Therefore the parity of $V$ is $(-1)^n$. Also $\phi(V) = (1)$, the identity permutation. Hence $\overline{sgn}(V)=1$. These parities, along with the relation between perfect shuffles and unshuffles found in Proposition~\ref{connection}, help us construct Table~\ref{signs} for the parities of $L, R, \phi(L),$ and $\phi(R)$, which we will use frequently.

\begin{table}
\begin{center}
\setlength{\arrayrulewidth}{0.5mm}
\setlength{\tabcolsep}{10pt}
\renewcommand{\arraystretch}{1.5}

\begin{tabular}{ |p{2.7cm}|p{1.0cm}|p{1.0cm}|p{1.0cm}|p{1.0cm}|  }
\hline
 & $L$& $R$& $\phi(L)$& $\phi(R)$ \\
\hline
$n \equiv 0\pmod{4}$ & $~~1$ & $~~1$ & $~~1$ & $~~1$ \\
$n \equiv 1\pmod{4}$ & $~~1$ & $-1$ & $~~1$ & $~~1$ \\
$n \equiv 2\pmod{4}$ & $-1$ & $-1$ & $-1$ & $~~1$ \\
$n \equiv 3\pmod{4}$ & $-1$ & $~~1$ & $~~1$ & $-1$ \\
\hline
\end{tabular}
\end{center}
\caption{Parities of $L$, $R$, $\phi(L)$, $\phi(R)$ for congruence classes of $n$ modulo $4$.}\label{signs}
\end{table}

The kernel of the homomorphism $\phi:B_n \longrightarrow S_n$ is generated by all transpositions $(x,x')$ where $x \in \{0, 1, \ldots, n-1\}$ and $x' = 2n-1 - x$. These transpositions commute with each other. Hence $ker(\phi)$ is a group of order $2^n$. We denote the restriction of the homomorphism $\phi$ to $G$ by $\phi|_G$, and we denote the kernel of $\phi|_G$ by $K$. The following lemma finds the order of $K$. 

\begin{lemma}\label{K}
Let $n> 1$ be such that $n$ is not a power of 2 and $n\neq 6, 12$, and let $K$ denote the kernel of the homomorphism $\phi|_G$.
    If $n\equiv 0 \pmod{4}$, then $|K| = 2^{n-1}$. Otherwise, $|K| = 2^{n}$.
\end{lemma}
\begin{proof}    
     We first construct a permutation $f\in B_n$ as follows:
       $$ f(i) = 
\begin{cases}
L(i) & \text{if $i$ is even and $0\leq i \leq n-1$}\\
& \text{ \quad or if $i$ is odd and $n\leq i \leq 2n-1$ }\\
L(i)' & \text{if $i$ is odd and $0\leq i \leq n-1$}\\
&\text{  \quad or if $i$ is even and $n\leq i \leq 2n-1$ }
\end{cases}
    $$ 
     Our movitation for creating this permutation is that $f$ induces the same permutation as $L$ on the centrally symmetric pairs, so $\phi(f) = \phi(L)$, but unlike $L$, the permutation $f$ leaves the set $\{0, 1, \ldots, n-1\}$ invariant (hence, of course, $f$ also leaves the set $\{n, n+1, \ldots, 2n-1\}$ invariant). To verify this, consult Equation~\ref{L} for $L(i)$. We will make use of $f$ throughout.
     
     Now suppose that $n \equiv 0, 1,$ or $ 3 \pmod{4}$. In this case, Table~\ref{signs} tells us $\phi(L)$ is an {\em even} permutation. Denote this permutation via $\sigma\in S_n$. It follows that $f = \sigma\sigma'$. By Lemma~\ref{G*}, since $\sigma$ is an even permutation, we know that $f$ is an element of $G^*$. Therefore $f$ can be realized using left and right shuffles.

    Now consider the permutation: $f^{-1}L$. 
Written in cycle notation, we have:
$f^{-1}L = (11')(33')\cdots(kk')$, where $k=n-1$ if $n$ is even and $k=n-2$ if $n$ is odd.

In the special case that $n=3,4,$ or 5, we have all the ingredients we need to proceed, but for $n>5$, we need to define one extra ingredient: 
$g = (0,1)(3,5)(0',1')(3',5')$. 
Notice that by Lemma~\ref{G*}, we know $g \in G^*$, so $g$ can be realized with left and right shuffles, as well. With these ingredients ready, we construct the permutation $h$: 
    $$
h =
\begin{cases}
f^{-1}L=(1 1') &\textrm{ if } n=3\\
f^{-1}L=(1 1')(3 3') &\textrm{ if } n=4, 5\\
(f^{-1}L)g(f^{-1}L)g =(0 0')(1 1') &\textrm{ if }n> 5
\end{cases}  
$$    Observe that because $f, L, g \in G$, this permutation $h$ is also an element of $G$. Moreover, $h\in K$ because $h$ is a product of transpositions of the form $(xx')$.

    Now suppose $n=3$. Conjugating $h=(1 1')$ by elements of $G^*$, we generate $(0 0'), (1 1'),$ and $(2 2')$, all of which are in $K$. Together these three elements will generate a set of 8 elements. Therefore $|K|\geq 2^3.$ But on the other hand, we know $|K|\leq 2^3$ since $K$ is a subgroup of the kernel of $\phi$, so then $|K|=2^3$. This concludes the argument for $n=3.$ 
    
    Next suppose $n\geq 4$ (and we continue to assume $n \equiv 0, 1,$ or $ 3 \pmod{4}$). Conjugating $h$ by elements of $G^*$, it follows that all elements of the form $(yy')(zz')$ where $y,z \in \{0,1,\ldots, n-1\}$ are in $K$. Multiply these elements of the form $(yy')(zz')$  together in all possible ways, and we produce any product of an {\em even} number of these transpositions. The total number of such products (all of which belong to $K$) is:
    $$\binom{n}{0}+\binom{n}{2}+ \binom{n}{4}+\cdots ~~ = 2^{n-1}. $$
    Hence $|K|\geq 2^{n-1}$ if $n \equiv 0, 1, 3 \pmod{4}$ and $n\geq 4$. 
    
    If $n \equiv 2 \pmod{4}$, we can use a similar argument to above to prove that $|K|\geq 2^{n-1}$. In this case, $\phi(R)$ is an even permutation instead of $\phi(L)$, and so the argument proceeds by switching $L$ to $R$ and making related adjustments. We leave these details to the reader and move forward assuming that $|K|\geq 2^{n-1}$ for all $n$.  

If $n \equiv 0 \pmod{4}$, then $L$ and $R$ are both even permutations, so all permutations in $K$ are even. Since $K$ is a subgroup of the kernel of $\phi$ which contains odd permutations and has order $2^n$, we know $|K| < 2^n$. On the other hand, as we exhibited above, $|K|\geq 2^{n-1}$. Therefore $|K|=2^{n-1}$, as desired.

    If $n \equiv 3 \pmod{4}$, then $L$ is an odd permutation and $f^{-1}L$ is an odd permutation in $K$. Multiplying $f^{-1}L$ together with the elements in the set of even permutations in $K$ we constructed above, we generate another $2^{n-1}$ unique elements in $K$. Therefore $|K|\geq 2^{n-1}+2^{n-1} = 2^n$. On the other hand, we know $|K| \leq 2^n$ because $K$ is a subgroup of the kernel of $\phi$, so the result follows.

    Finally, suppose that $n \equiv 1 \text{ or } 2 \pmod{4}$. In this case, $R$ is an odd permutation and $\phi(R)$ is even. Similar to before, let $k\in G^*$ such that $\phi(k) = \phi(R)$. Then $kR^{-1}$ is an odd permutation in $K$. As in the previous case, it follows that $|K|=2^n.$
\end{proof}

We are now ready to determine the group structure of $G = \langle L, R\rangle$ for all $n>1$ such that $n$ is not a power of 2 and $n\neq 6, 12$.
\begin{theorem}\label{main}
Suppose a deck has $2n$ cards and let $G = \langle L, R\rangle$. 
\begin{enumerate}[(a)]
\item If $n \equiv 0\pmod{4}$, $n > 12$, and $n$ is not a power of 2, $G = \langle I, O\rangle$.
\item If $n \equiv 1\pmod{4}$ and $n>1$, $G = \langle I, O\rangle$.
\item If $n \equiv 2\pmod{4}$ and $n > 6$, $G = \langle I, O\rangle = B_n$.
\item If $n \equiv 3\pmod{4}$, $G = B_n$.
\end{enumerate}
\end{theorem}
\begin{proof}
In all cases, we know that $A_n \subseteq \phi(G^{*})\subseteq \phi(G)$ by Lemma~\ref{G*}, and we will frequently use this fact.

We begin with (a). Note that both $\phi(L)$ and $\phi(R)$ are even permutations for $n \equiv 0\pmod{4}$ (Table~\ref{signs}). So, we conclude that $\phi(G) = A_n$ since $A_n \subseteq \phi(G^{*})\subseteq \phi(G)$. By the First Isomorphism Theorem and Lemma~\ref{K}, we have that $\lvert G \rvert = \lvert A_n \rvert \lvert K \rvert=n!\cdot2^{n-2}$.

We next prove that $G=\langle L, R\rangle \subseteq \langle I, O\rangle$ when $n \equiv 0\pmod{4}$. Because $L = VI^{-1}$ and $R = VO^{-1}$, it suffices to show that $V \in \langle I, O\rangle$. 
By Theorem~\ref{perfectgroups}, $\langle I, O\rangle$ is the intersection of the kernels of $sgn$ and $\overline{sgn}$. We already discussed that the parity of $V$ is $(-1)^{n}$, and since $n$ is even, $sgn(V) = 1$. We also know $\overline{sgn}(V)=1$. Therefore it follows that $V \in \langle I, O\rangle$, and so $\langle L, R\rangle \subseteq \langle I, O\rangle$. Because the two groups have the same order, we conclude $G = \langle L, R\rangle= \langle I, O\rangle$.

For (b), we can use a similar argument to that of (a) to show that $\lvert G \rvert = \lvert A_n \rvert \lvert K \rvert$. By Lemma~\ref{K}, $\lvert G \rvert = n!\cdot2^{n-1}$. Now Theorem~\ref{perfectgroups} tells us that for $n\equiv 1\pmod{4}$, the group $\langle I, O\rangle$ is equal to the the kernel of $\overline{sgn}$. Since $\overline{sgn}(V)=1$, we know $V \in \langle I, O\rangle$. So, $G=\langle L, R\rangle \subseteq \langle I, O\rangle$. Because two groups have the same order, it must be the case that $G = \langle L, R\rangle= \langle I, O\rangle$.

We now prove (c). Note that $\phi(L)$ is odd when $n \equiv 2\pmod{4}$. So, we can conclude that $\phi(G) = S_n$ since $A_n$ together with $\phi(L)$ will generate all permutations of $S_n$. 
By the First Isomorphism Theorem and Lemma~\ref{K}, we have that $\lvert G \rvert = 	\lvert S_n \rvert	\lvert K \rvert= n!\cdot2^{n}$.
Now $\langle L, R\rangle$ is a subgroup of $B_n$, but since the two groups have the same order, we conclude $G = \langle L, R\rangle= \langle I, O\rangle = B_n$.

Our proof for (d) is similar to that of (c). In this case, it is $\phi(R)$ that is odd instead of $\phi(L)$ (Table~\ref{signs}), but in the same way we can conclude that $\phi(G) = S_n$. The First Isomorphism Theorem and Lemma~\ref{K} tell us that $\lvert G \rvert = \lvert S_n \rvert \lvert K \rvert= n!\cdot2^{n}$.
We know that $G=\langle L, R\rangle \subseteq B_n$, but since the orders of the two groups are equal, we get our desired conclusion of $G = B_n$.
\end{proof}

\begin{example}
    The smallest $n$ where the perfect shuffle group and unshuffle group differ is when $n=3$ (a deck of 6 cards). A deck of 6 cards has $6! = 720$ possible arrangements total. Perfect shuffles and unshuffles can realize only 24 and 48 card arrangements, respectively. 
    
    In particular,  Theorem~\ref{perfectgroups} tells us that the group $\langle I, O\rangle$ is equal to the kernel of $sgn\overline{sgn}$ which is a group of order 24 isomorphic to $S_4$, the symmetric group on 4 elements. On the other hand, Theorem~\ref{main} tells us the group $\langle L, R\rangle$ is $B_3$, the group of all centrally symmetric permutations of 6 elements. This group has order 48 and is isomorphic to the direct product $S_4 \times \mathbb{Z}_2$. 
\end{example}

We have restricted our focus to perfect shuffles and unshuffles here, but a wide variety of different shuffling techniques exist which provide mathematical diversion~\cite{Bayer, Butler, Johnson, Ledet, Medvedoff, MorrisHartwig}. Three great expository books that discuss these ideas in depth and describe fun mathematical card tricks to impress your friends and family are~\cite{diaconis2011,Morris2, mulcahy}.

\bibliographystyle{plain}
\bibliography{biblio}

\begin{thebibliography}{10}

\bibitem{expose}
Anonymous.
\newblock {\em A grand expose of the science of gambling}.
\newblock New York : F. A. Brady, 1860.

\bibitem{Bayer}
Dave Bayer and Persi Diaconis.
\newblock Trailing the dovetail shuffle to its lair.
\newblock {\em Ann. Appl. Probab.}, 2(2):294--313, 1992.

\bibitem{braue}
Frederick Braue and Jean Hugard.
\newblock {\em Expert card technique}.
\newblock Martino Publishing, 1940.

\bibitem{Butler}
Steve Butler, Persi Diaconis, and Ron Graham.
\newblock The mathematics of the flip and horseshoe shuffles.
\newblock {\em Amer. Math. Monthly}, 123(6):542--556, 2016.

\bibitem{diaconis2011}
P.~Diaconis and R.~Graham.
\newblock {\em Magical Mathematics: The Mathematical Ideas That Animate Great
  Magic Tricks}.
\newblock Princeton University Press, 2011.

\bibitem{Diaconis}
Persi Diaconis, R.~L. Graham, and William~M. Kantor.
\newblock The mathematics of perfect shuffles.
\newblock {\em Adv. in Appl. Math.}, 4(2):175--196, 1983.

\bibitem{DiaconisGraham}
Persi Diaconis and Ron Graham.
\newblock The solutions to {E}lmsley's problem.
\newblock {\em Math Horizons}, 14(3):22--27, 2007.

\bibitem{Ensley}
Doug Ensley.
\newblock Unshuffling for the imperfect magician.
\newblock {\em Math Horizons}, 11(3):13--16, 2004.

\bibitem{Green}
J.~H. Green.
\newblock {\em An exposure of the arts and miseries of gambling}.
\newblock Philadelphia, Pennsylvania: G. B. Zieber, 1847.

\bibitem{Johnson}
Samuel Johnson, Lakshman Manny, Cornelia~A. Van~Cott, and Qiyu Zhang.
\newblock A look at generalized perfect shuffles.
\newblock {\em Involve}, 14(5):813--828, 2021.

\bibitem{Jordan}
Charles~T. Jordan.
\newblock {\em Thirty card mysteries}.
\newblock Pengrove, CA, 1919.

\bibitem{Ledet}
Arne Ledet.
\newblock The {M}onge shuffle for two-power decks.
\newblock {\em Math. Scand.}, 98(1):5--11, 2006.

\bibitem{Medvedoff}
Steve Medvedoff and Kent Morrison.
\newblock Groups of perfect shuffles.
\newblock {\em Math. Mag.}, 60(1):3--14, 1987.

\bibitem{Morris}
S.~Brent Morris.
\newblock The basic mathematics of the faro shuffle.
\newblock {\em Pi Mu Epsilon J.}, 6:85--92, 1975.

\bibitem{Morris2}
S.~Brent Morris.
\newblock {\em Magic tricks, card shuffling and dynamic computer memories}.
\newblock MAA Spectrum. Mathematical Association of America, Washington, DC,
  1998.
\newblock With an introduction by Martin Gardner.

\bibitem{MorrisHartwig}
S.~Brent Morris and Robert~E. Hartwig.
\newblock The generalized faro shuffle.
\newblock {\em Discrete Math.}, 15(4):333--346, 1976.

\bibitem{mulcahy}
C.~Mulcahy.
\newblock {\em Mathematical Card Magic: Fifty-Two New Effects}.
\newblock Taylor \& Francis, 2013.

\bibitem{Ramnath}
Sarnath Ramnath and Daniel Scully.
\newblock Moving card {$i$} to position {$j$} with perfect shuffles.
\newblock {\em Math. Mag.}, 69(5):361--365, 1996.

\bibitem{swinford1}
Paul Swinford.
\newblock {\em Faro fantasy}.
\newblock The Haley Press, Connersville, Indiana, 1968.

\bibitem{swinford2}
Paul Swinford.
\newblock {\em More Faro fantasy}.
\newblock Self published, 1971.

\end{thebibliography}

\end{document}